\documentclass{article}

\usepackage{arxiv}

\usepackage[utf8]{inputenc} 
\usepackage[T1]{fontenc}    
\usepackage{hyperref}       
\usepackage{url}            
\usepackage{booktabs}       
\usepackage{amsfonts}       
\usepackage{nicefrac}       
\usepackage{microtype}      
\usepackage{lipsum}
\usepackage{amsmath}

\usepackage{amsfonts,amssymb,stmaryrd}
\usepackage{graphicx}
\usepackage{amsthm}
\usepackage{xcolor}

\newcommand{\R}{\mathbb{R}}

\newtheorem{theorem}{Theorem}[section]

\newtheorem{corollary}{Corollary}[section]

\newtheorem{definition}{Definition}[section]
\newtheorem{remark}{Remark}[section]
\newtheorem{example}{Example}[section]

\usepackage{subcaption}
\usepackage{cite}
\usepackage{authblk}
\newcommand{\p}{\partial}
\newcommand{\bb}{\begin{equation}}
\newcommand{\ee}{\end{equation}}
\newcommand{\ba}{\begin{array}}
\newcommand{\ea}{\end{array}}
\newcommand{\f}{\frac}
\usepackage[all]{xy}
\newcommand{\ds}{\displaystyle}
\newcommand{\eps}{\varepsilon}
\newcommand{\al}{\alpha}

\newcommand{\uu}{{\bf u}}
\newcommand{\s}{\mathbb{S}}

\newcommand{\B}{{\cal B}}

\numberwithin{equation}{section}

\usepackage{ulem}

\title{Unique continuation results for abstract quasi-linear evolution equations in Banach spaces}

\author[1 \thanks{igor.freire@ufscar.br and igor.leite.freire@gmail.com}] {Igor Leite Freire}
\affil[1]{Departamento de Matemática, Universidade Federal de São Carlos\\

Rodovia Washington Luís, Km 235, 13565-905,
São Carlos-SP, Brasil}
 
\begin{document}
\maketitle

\begin{abstract}

Unique continuation properties for a class of evolution equations defined on Banach spaces are considered from two different point of views: the first one is based on the existence of conserved quantities, which very often translates into the conservation of some norm of the solutions of the system in a suitable Banach space. The second one is regarded to well-posed problems. Our results are then applied to some equations, most of them describing physical processes like wave propagation, hydrodynamics, and integrable systems, such as the $b-$; Fornberg-Whitham; potential and $\pi-$Camassa-Holm; generalised Boussinesq equations; and the modified Euler-Poisson system.

\end{abstract}

{\bf MSC classification 2010:} 35A01, 74G25, 37K40, 35Q51.

\keywords{Conserved quantities \and Unique continuation of solutions \and Local well-posedness}
\newpage
\tableofcontents
\newpage

\section{Introduction}\label{sec1}

In this paper we consider the problem of unique continuation of solutions for systems of evolution equations
\bb\label{1.0.1}
u^i_t+A^i_ju^j=g^i,\quad 1\leq i\leq m,
\ee
where $m$ is a fixed positive integer; $\uu=(u^1,\cdots,u^m)$ is a vector function of real variables $t$ and $x$, such that for each fixed $t$ for which the solution exists, $\uu(t,\cdot)$ belongs to a certain Banach space; $A^i_j=A^i_j([\uu])$ are $m^2$ linear operators; $g=g([\uu])$ is a member of some Banach space for each fixed $t$, where $g=(g^1,\cdots,g^m)$; and $[\uu]$ denotes dependence on $\uu$ and its ($x-$)derivatives up to a finite, but arbitrary, order. The Einstein summation convention over repeated indices is presupposed in \eqref{1.0.1}. In case $m=1$, we simply denote $u$ instead of $\uu$. 

System \eqref{1.0.1} can be conveniently rewritten in a shorter form
\bb\label{1.0.2}
{\bf u}_t+A([{\bf u}]){\bf u}=g([{\bf u}]),
\ee
where $A([\uu])=(A^i_j([\uu]))$ is a (matrix of) linear operator(s). We would like to point out that under mild conditions on $A(\cdot)$ and $g(\cdot)$ one can establish local well-posedness of the equation above subject to $\uu(0,\cdot)=\uu_0(\cdot)$ using Kato's approach \cite{kato}.

Important equations describing waves propagating in a physical media fall into the class \eqref{1.0.2}, or can be taken as such. To restrict ourselves to a single example, we mention the non-local form of the Camassa-Holm equation \cite{chprl}
\bb\label{1.0.3}
u_t+uu_x=-\p_x\Lambda^{-2}\Big(u^2+\f{1}{2}u_x^2\Big),
\ee
which is a famous model in hydrodynamics.

In this work we present two different approaches for studying unique continuation of solutions of equations of the type \eqref{1.0.2}:
\begin{itemize}
    \item The first one does not necessarily presume uniqueness of solutions, but requires the existence of a conserved quantity satisfying certain mild conditions;
    \item The second one makes use of local well-posedness, but does not require any conserved quantity.
\end{itemize}

In some sense, our results could be seen as unique continuation counter-parts for some classes of abstract evolution equations considered by Kato \cite{kato}: while his theory can be applied to a large class of equations of the type \eqref{1.0.2}, our findings require more conditions on the map $g(\cdot)$, stronger than those requested by Kato for proving local well-posedness for Cauchy problems involving \eqref{1.0.2}. Despite being applicable to a class of problems smaller than that covered by Kato's local well posedness approach, the tools introduced in the present paper can still be applied to a considerably wide class of equations describing physical processes in hydrodynamics, wave propagation and integrable systems. 

Let us make an overview of the manuscript's structure: in the next section we present our main results, namely, our approaches for dealing with unique continuation properties for the solutions of \eqref{1.0.2}, which we call {\it conserved quantity approach} and {\it local well-posedness approach}. Then, in section \ref{sec3} we illustrate our results by applying them to some equations, mostly of them coming from hydrodynamics and integrable systems theory.

The machinery developed in the present work enables us to treat problems of unique continuation of solutions for other classes of equations beyond \eqref{1.0.2} with some customisation, and depending on the situation, regardless the order of the derivative with respect to $t$. As an example, the tools introduced here give us condition to prove unique continuation results for the generalised Boussinesq equation
\bb\label{1.0.4}
u_{tt}=\p_x^2f(u)+u_{ttxx},
\ee
which was considered by Constantin and Molinet \cite{moli}. This is done in section \ref{sec4}. Such an equation, for certain choices of the function $f(\cdot)$, can  describe nonlinear waves in weakly dissipative media \cite{clark}. 

Another example is given in section \ref{sec5}, where we use the ideas given in section \ref{sec2} to prove that if $(u,\rho)$ is a conservative solution (this notion will be given in section \ref{sec5}) of the modified Euler-Poisson system
\bb\label{1.0.5}
\left\{
\ba{l}
\rho_t+(u\rho)_x=0,\\
\\
u_t-u_{txx}+uu_x-3u_xu_{xx}-uu_{xxx}+\rho_x=0,
\ea
\right.
\ee
vanishing on an open set $\emptyset\neq\Omega\subseteq\R^2$, then $\rho\equiv0$ and $u$ is necessarily a solution of the inviscid Burgers equation 
$$u_t+uu_x=0.$$

Our results are discussed in section \ref{sec6}, while our conclusions are presented in section \ref{sec7}.

\section{Main results}\label{sec2}

\subsection{Notation and conventions}

The set of continuous operators between two Banach spaces $\B_1$ and $\B_2$ is denoted by ${\cal L}(\B_1;\B_2)$, while $\B_1\hookrightarrow\B_2$ means that the Banach space $\B_1$ is continuously embedded in the Banach space $\B_2$. 

By $X$ we denote the real line $\R$ or the circle $\s$, that can be identified with the interval $[0,1)$. For a function $u$ with two variables, partial derivatives with respect to the first argument will be denoted by $u_t$, whereas partial derivatives of $u$ with respect to its second argument will be often denoted by $u_x$ or, eventually, by $\p_xu$.

The convolution between two functions $f$ and $g$ is denoted by $f\ast g$; $H^s(X)$, $s\in\R$, is the usual Sobolev space of order $s$; $\Lambda^2:=1-\p_x^2$, and its inverse $\Lambda^{-2}f=g\ast f$ is given by
$$
g(x)=\left\{\ba{ll}
\ds{\f{e^{-|x|}}{2}},&\text{if}\,\,X=\R,\\
\\
\ds{\frac{\cosh(x-\lfloor x \rfloor - 1/2)}{2\sinh(1/2)}}, & \text{if}\,\,X=\mathbb{S}.
\ea
\right.
$$

Very often we make use of the identity $\p_x^2\Lambda^{-2}=\Lambda^{-2}-1$ whenever such operator appears and is defined. For $s\geq t$, then $H^s(X)\hookrightarrow H^{t}(X)$, and $H^s(X)$ is an algebra as soon as $s>1/2$, that is, if $u,v\in H^s(X)$, then $uv\in H^s(X)$. Additionally, given $u,v\in H^s(X)$, with $s>1/2$, then $\|uv\|_{H^{s-1}(X)}\leq c\|u\|_{H^s(X)}\|v\|_{H^{s-1}(X)}$, for some constant $c>0$ depending on $s$. For further details, see \cite[Chapter 4]{taylor} and \cite[Appendix]{kato}.

\subsection{General assumptions}

Throughout the paper we assume the following global conditions on the system \eqref{1.0.2} and its solutions $\uu$:

\begin{enumerate}
    \item[{\bf G0}] $\uu$ is defined on $[0,T)\times X$, for some $T>0$;
    \item[{\bf G1}] There exist two Banach spaces $\B_1$ and $\B_2$, such that $\B_1\hookrightarrow\B_2$, and $\uu\in C^0([0,T);\B_1)\cap C^1((0,T);\B_2)$;
    \item[{\bf G2}] $A([\uu])\in {\cal L}(\B_1;\B_2)$, for each fixed $t$;
    \item[{\bf G3}] $g([\uu])\in \B_2$, for each fixed $t$. Moreover, $\uu\mapsto g([\uu])$ is continuous.
    
\end{enumerate}

Henceforth, we simply write $\uu$ instead of $[\uu]$ in order to make the notation cleaner and more concise. Moreover, in some parts of the presentation we replace $[\uu]$ by $\uu(t,x)$ in order for the independent variables to be emphasised.

We would like to make a small digression and present some short, but useful, comments about the general conditions above.

\begin{remark}\label{rem2.1}
Condition {\bf G0} does not necessarily request, and nor imply, uniqueness of the solution, although {\bf G1} may lead to such under mild conditions, see \cite[Theorem 6]{kato}.
\end{remark}

\begin{remark}\label{rem2.2}
A solution $u$ satisfying condition {\bf G1} need not to be continuous, although it may imply that $u$ is continuously differentiable provided that members of the space ${\cal B}_2$ have enough regularity, e.g, ${\cal B}_2=H^s$, for $s$ large enough.
\end{remark}

\begin{remark}\label{rem2.3}
For each $t$ such that the solution exists, conditions {\bf G2} and {\bf G3} say that $A(\uu(t,\cdot))$ is a bounded linear operator between the Banach spaces $\B_1$ and $\B_2$, whereas $g(\uu(t,\cdot))\in{\cal B}_2$. Moreover, for $t>0$, we have $g(\uu(t,\cdot))=\uu_t(t,\cdot)+A(\uu(t,\cdot))\uu(t,\cdot)$. This relation plays vital importance in the demonstration of our main results.
\end{remark}

\begin{remark}\label{rem2.4}
We observe that $g$ must vanish at $0$ for sake of consistency.
\end{remark}

Below we present some examples illustrating the conditions above. 

\begin{example}\label{ex2.1}
Let $u\in C^0([0,T);H^{s}(X))\cap C^1((0,T);H^{s-1}(X))$, $A(u)=u\p_x$, ${\cal B}_1=H^s(X)$, $\B_2=H^{s-1}(X)$, and $s>3/2$. Then conditions {\bf G0}--{\bf G2} are clearly satisfied and, in particular, $A(u)u=uu_x$.

In addition, the fact that the operator $\p_x\Lambda^{-2}$ applies $H^s(X)$ into $H^{s+1}(X)$ implies 
$$g(u)=\p_x\Lambda^{-2}\Big(au+bu^2+c u_x^2\Big)\in H^{s-1}(X),$$
for all $a,b,c\in\R$, since $H^{s}(X)\hookrightarrow H^{s-1}(X)$, and hence, ${\bf G3}$ is satisfied.
\end{example}

\begin{example}\label{ex2}
Let $A(u)=u_x\p_x$ and $g(u)=\Lambda^{-2}(u_x^2+u_{xx}^2/2)$. Taking  $u\in C^0([0,T);H^s(\R))\cap C^1((0,T);H^{s-1}(\R))$, with $s>5/2$, we see that {\bf G0}--{\bf G3} are satisfied with $\B_1=H^s(X)$, $\B_2=H^{s-1}(X)$, $s>5/2$.
\end{example}

Let $\uu$ be a solution of \eqref{1.0.2}. Fixed $t\in[0,T)$, we define  
\bb\label{2.0.1}
F_t(x):=g(\uu(t,x)),
\ee
and for this reason, in some parts of the paper we write $F_t=(F^1_t,\cdots,F_t^m)$. Observe that \eqref{2.0.1} defines a family of (vector) functions 
$(F_t)_{t\in[0,T)}\subseteq{\cal B}_2$.
\subsection{Conserved quantity approach}\label{sec2.3}

In addition to the global conditions ${\bf G0}-{\bf G3}$, throughout this subsection we also assume the following:
    \begin{enumerate}
    \item[{\bf C0}] There exists a continuous, non-negative or non-positive, real valued function $h$, depending on $u$ and its derivatives of arbitrary, but finite order, which we simply write as $h=h(\uu(t,x))$, such that
    \bb\label{2.1.1}
    {\cal H}(t)=\int_X h(\uu(t,x))dx
    \ee
    is constant. Additionally, we assume that $h(y)=0$ if and only if $y=0$.
    \end{enumerate}
    
The quantity ${\cal H}(t)$ and the function $h(\uu(t,x))$ in \eqref{2.1.1} are called {\it conserved quantity} and {\it conserved density}, respectively, whereas in the jargon of analysis of PDEs, very often ${\cal H}(t)$ is called {\it conservation law}. Condition ${\bf C0}$ is nothing but the assumption of the existence of a conserved quantity for solutions of \eqref{1.0.2}. 

\begin{remark} Even though we could include explicit dependence on the independent variables without changing the results, we prefer the restriction of the conserved densities to those only depending on $\uu$ and its derivatives of some finite order for sake of simplicity. In case these other variables are taken into account, we should assume that $h(t,x,[\uu])$ vanishes if and only if $\uu=0$. Moreover, note that $[\uu]\equiv0$ if and only if $\uu\equiv0$. Finally, it is worth mentioning that we identify a function $\uu$ with its zeroth-order derivative.

\end{remark}

Some examples illustrating {\bf C0} are:
\begin{itemize}
    \item Let $u=u(t,x)$ be a real valued function and $h(u)=u$. If $u$ is non-negative, then \eqref{2.1.1} implies that $\|u\|_{L^1(\R)}$ is a conserved quantity.
    \item Still assuming that $u=u(t,x)$ and $\rho=\rho(t,x)$ are real valued functions, then $h(u)=u^2+u_x^2+\rho^2$ is a density implying that the product norm $\|u\|_{H^1(\R)}+\|\rho\|_{L^2(\R)}$ is conserved.
\end{itemize}

Above we assume that all integrals involved are finite. More examples can be found and are discussed in detail in section \ref{sec3}. Henceforth, $a$ and $b$ denote distinct real numbers satisfying $a<b$.

\begin{theorem}\label{teo2.1}
Let $t_0\in(0,T)$, $a,b\in X$, and $\uu$ be a solution of \eqref{1.0.2}. Suppose that $\uu$ and $g(\cdot)$ satisfy ${\bf C0}$ and the conditions
\begin{enumerate}
    \item[{\bf C1}] $\uu(t_0,x)=0$, $\uu_t(t_0,x)=0$, for $x\in[a,b]$;
    \item[{\bf C2}] $g(\uu(t_0,x))=0$, $x\in[a,b]$, implies $\uu(t_0,\cdot)\equiv0$.
\end{enumerate}
Then $\uu\equiv0$.
\end{theorem}

\begin{proof}
From \eqref{1.0.2} (see also Remark \ref{rem2.3}), we have
\bb\label{2.1.2}
g(\uu(t,x))=\uu_t(t,x)+A(\uu(t,x))\uu(t,x).
\ee
As a consequence, restricting $x$ to $[a,b]$, and taking $t=t_0$, from ${\bf C1}$ we conclude that $g(\uu(t_0,x))=0$, whereas ${\bf C2}$ implies that $\uu(t_0,x)=0$, $x\in X$. The conserved quantity \eqref{2.1.1} then implies ${\cal H}(t)={\cal H}(t_0)=0$, for any $t\in(0,T)$. Due to the fact that $h(\cdot)$ is either non-negative or non-positive, the condition
$$
\int_Xh(\uu(t,x))dx=0,
$$
implies that $h(\uu(t,\cdot))=0$, that is, $\uu(t,\cdot)=0,\,\,t\in[0,T)$, since $h(y)=0$ if and only if $y=0$.
\end{proof}

\begin{corollary}\label{cor2.1}
If $\uu$ is a solution for \eqref{1.0.2} such that ${\bf C0}$ and ${\bf C2}$ hold, and there exists an open set $\Omega\subseteq[0,T)\times X$ for which $\uu\Big|_{\Omega}\equiv0$, then $\uu\equiv0$.
\end{corollary}

\begin{proof}
Let us take $t_0\in(0,T)$ and $a,b\in X$ such that $\{t_0\}\times[a,b]\subseteq\Omega$. Then condition ${\bf C1}$ is satisfied and the result is a foregone conclusion of Theorem \ref{teo2.1}.
\end{proof}

\begin{theorem}\label{teo2.2}
Let $\uu=(u^1,\cdots,u^m)$, $F_t=(F_t^1,\cdots,F_t^m)$, $t\in[0,T)$, where $F_t$ is given by \eqref{2.0.1}, and suppose that ${\bf C0}$ holds. Assume that there exists $t_0\in(0,T)$, $a,b\in X$, and some $i\in\{1,\cdots,m\}$ such that:
\begin{enumerate}
    \item[{\bf C3}] $\uu(t_0,x)=0$, $x\in[a,b]$;
    \item[{\bf C4}] $\uu_t(t_0,a)=\uu_t(t_0,b)$;
    \item[{\bf C5}] $\p_x F_{t_0}^i(\cdot)$ is either non-positive or non-negative on $(a,b)$;
    \item[{\bf C6}] $\p_x F^i_{t_0}(x)=0$, for $x\in(a,b)$, implies $\uu(t_0,\cdot)=0$.
\end{enumerate}
Then $\uu\equiv0$.
\end{theorem}

\begin{proof}
By the Fundamental Theorem of Calculus, we have
\bb\label{2.1.3}
F_{t_0}^i(b)-F_{t_0}^i(a)=\int_{a}^b\p_x F^i_{t_0}(s)ds.
\ee

Conditions ${\bf C3}$ and ${\bf C4}$, jointly with \eqref{2.1.2}, imply $g(\uu(t_0,a))=g(\uu(t_0,b))=0$. By \eqref{2.0.1} and \eqref{2.1.2} we conclude that $F_{t_0}(a)=F_{t_0}(b)$, which, jointly with \eqref{2.1.3} tell us that
$$
\int_{a}^b\p_x F_{t_0}^i(s)ds=0.
$$

By {\bf C5,} $\p_x F^i_{t_0}(\cdot)$ is non-positive or non-negative. Hence, the equality above implies that $\p_x F_{t_0}^i(x)=0$, $x\in(a,b)$, while {\bf C6} tells that $\uu(t_0,x)=0$, $x\in X$. As consequence, from \eqref{2.1.1} we arrive at ${\cal H}(t_0)=0$, which implies $\uu(t,\cdot)=0$ in view of the invariance of the conserved quantity.
\end{proof}

\subsection{Local well-posedness approach}\label{sec2.4}

The unique continuation results proved before were based on the existence of a conserved quantity as stated by condition {\bf C0}. We now look to the problem using the following different conditions.
\begin{itemize}
    \item[{\bf C7}] The Cauchy problem
    \bb\label{2.2.1}
    \left\{
    \ba{lcl}
    \uu_t+A(\uu)\uu&=&g(\uu(t,x)),\\
    \\
    \uu(0,x)&=&\uu_0(x),
    \ea
    \right.
    \ee
for some $u_0\in {\cal B}_1$, is locally well-posed, so that we can at least assure the existence of a unique local solution $u$ defined on $[0,T)\times X$, for some lifespan $T>0$.

\end{itemize}

It is worth mentioning that system \eqref{1.0.2} is invariant under translations in $t$, although its use may not be evident at first glance. Then we have the following crucial observation.

\begin{remark}\label{rem2.6}
Condition {\bf C7} has a powerful implication for the approach we are bound to present: if we know that a solution $\uu$ of \eqref{2.2.1} satisfies $\uu(t_0,x)=0$, for all $x\in X$ and some $t_0\in (0,T)$, then $v=\uu(t+t_0,x)$ is a solution of 
$$
    \left\{
    \ba{lcl}
    v_t+A(v)v&=&g(v(t,x)),\\
    \\
    v(0,x)&=&0.
    \ea
    \right.
$$

Clearly the function $v\equiv0$ is a solution of such a Cauchy problem. Since it has a unique solution, we are forced to conclude that $v(t,x)=0$ (and it can be taken as a global solution) and, again by uniqueness, the same applies to $\uu$.
\end{remark}

\begin{theorem}\label{teo2.3}
Let $\uu$ be a solution of \eqref{1.0.2} satisfying {\bf C7}. If, for some $t_0\in(0,T)$, $a,b\in X$, {\bf C1} and {\bf C2} as in Theorem \ref{teo2.1} are satisfied, then $\uu\equiv0$.
\end{theorem}

\begin{proof}
If {\bf C1} is satisfied, then $g(\uu(t_0,x))=0$, $x\in[a,b]$. By {\bf C2}, $\uu(t_0,x)=0$, $x\in X$. The result is then a consequence of Remark \ref{rem2.5}.
\end{proof}

\begin{theorem}\label{teo2.4}
Let $\uu$ be a solution of \eqref{1.0.2} satisfying {\bf C7}. If, for some $t_0\in(0,T)$, $a,b\in X$, {\bf C3}--{\bf C6} as in Theorem \ref{teo2.2} are satisfied, then $\uu\equiv0$.
\end{theorem}

\begin{proof}
The proof of theorem \ref{teo2.2} shows that $u(t_0,x)=0$. The result is again a consequence of Remark \ref{rem2.6}.
\end{proof}

\section{Applications}\label{sec3}

Here we illustrate our results by applying them to some equations. As a first application we specialise them to a general class of equations including relevant models in hydrodynamics and integrable systems, such as the Benjamin-Bona-Mahony (BBM) or the Camassa-Holm (CH) equations. 

\begin{theorem}\label{teo3.1}
Let $g=g(u)$ and $G=G(u)$ be continuous functions and suppose that $u$ is a solution of the equation
\bb\label{3.0.1}
u_t+g(u)=\p_x\Lambda^{-2}G(u).
\ee
Assume that $g=g(u)$ and $G=G(u)$ are functions depending on $u$ and its $x$ derivatives up to a finite, but arbitrary, order, such that $\p_x\Lambda^{-2}G$ is well defined; both vanishing when $u=0$; $G(\cdot)$ is either non-positive or non-negative; and the condition $G(u(t_0,x))=0$ implies that $u(t_0,x)=0$, $x\in X$. 

If there exists an open set $\Omega$ such that $u\big|_\Omega\equiv0$ and at least one of the conditions
\begin{itemize}
    \item ${\bf C0}$, or
    \item ${\bf C7}$ (replacing the equation in \eqref{2.2.1} by \eqref{3.0.1}),
\end{itemize}
are satisfied, then $u\equiv0$. In particular, the result holds if $g(u)=a(u)u$, where $a=a(u)$ is a linear operator.
\end{theorem}

\begin{remark}\label{rem3.1}
Note that \eqref{3.0.1} belongs to the class \eqref{1.0.2} with $m=1$. Actually, taking $A(u)=0$ and replacing $g(u)$ by $\p_x\Lambda^{-2}G(u)-g(u)$ we obtain \eqref{3.0.1} from \eqref{1.0.2}.
\end{remark}

\begin{proof}
Let $t_0$, $a$ and $b$ such that $\{t_0\}\times[a,b]\subseteq\Omega$, and define $F_t(x)=\p_x\Lambda^{-2}G(t,x)$. By \eqref{3.0.1} we conclude that $F_{t_0}(a)=F_{t_0}(b)$. By the Fundamental Theorem of Calculus, we have
$$
0=F_{t_0}(b)-F_{t_0}(a)=\int_X\Lambda^{-2}Gdx=\int_X (g\ast G)dx,
$$
which implies that $G(u(t_0,x))=0$, and thus $u(t_0,x)=0$. The result then follows the same steps as theorem \ref{teo2.1} or Remark \ref{rem2.6}, depending on the condition taken.
\end{proof}

In what follows we apply our tools to establish unique continuation results of solutions for some concrete equations. The first example is the $b-$equation, and unique continuation for its solutions was previously established in \cite{linares}, see also \cite{pri-jde,nilay}. We present it here to show the difference between the approaches introduced in the previous sections. The remaining examples of this section are original.

\subsection{The $b-$equation}\label{sec3.1}

Let us apply theorem \ref{teo3.1} to the $b-$equation
$
u_t-u_{txx}+(b+1)uu_x=bu_xu_{xx}+uu_{xxx},
$
$b\in(0,3]$, which for $b=2$ and $b=3$ reduces to the Camassa-Holm and Degasperis-Procesi equations, respectively, see \cite{chprl}. It can be rewritten as
\bb\label{3.1.1}
u_t+uu_x=-\p_x\Lambda^{-2}\Big(\f{b}{2}u^2+\f{3-b}{2}u_x^2\Big).
\ee

Example \ref{ex2.1} implies that conditions {\bf G0}--{\bf G3} are satisfied whenever $u\in C^0([0,T);H^s(X))\cap C^1((0,T);H^{s-1}(X))$, with $s>3/2$. If we suppose that $b\in(0,3]$ and there exists an open set $\Omega$ for which $u\big|_\Omega\equiv0$, we can then find  $\{t_0\}\times[a,b]\subseteq[0,T)\times X$ such that conditions ${\bf C3}$ and ${\bf C4}$ in Theorem \ref{teo2.2} hold. 

Defining
\bb\label{3.1.2}
F_{t_0}(x)=-\p_x\Lambda^{-2}\Big(\f{b}{2}u^2+\f{3-b}{2}u_x^2\Big)(t_0,x),
\ee
we easily conclude that $F'_t(x)$ is non-positive. Moreover, as soon as $x\in[a,b]$, from \eqref{3.1.2}, \eqref{3.1.1}, and the Fundamental Theorem of Calculus, we have
$$
0=F_{t_0}(b)-F_{t_0}(a)=-\int_{a}^b\Big(\f{b}{2}u^2+\f{3-b}{2}u_x^2\Big)(t_0,x)dx,
$$
that is
\bb\label{3.1.3}
u(t_0,x)=0,\,x\in X.
\ee

\begin{itemize}
    \item {\bf Use of conserved quantities.} It is well known that \eqref{3.1.1} has the quantity
\bb\label{3.1.4}
{\cal H}(t)=\int_X u(t,x)dx
\ee
conserved, e.g., see \cite[Theorem 2.1]{anco}. Assuming that $u\geq0$ or $u\leq0$, then \eqref{3.1.4} implies the conservation of $\|u(t,\cdot)\|_{L^1(X)}$, and \eqref{3.1.3} gives $0=\|u(t_0,\cdot)\|_{L^1(X)}=\|u(t,\cdot)\|_{L^1(X)}$, which implies that $u\equiv0$.

Restricting particularly to the Camassa-Holm equation, it has the conserved quantity
$${\cal H}_1(t)=\int_{X}(u^2+u_x^2)dx=\|u\|^2_{H^1(X)}$$
conserved, see also \cite[Theorem 2.1]{anco} which implies that if ${\cal H}_1(t)$ vanishes at a point $t=t_0$, then is vanishes at all, as well as $u$. 

\item {\bf Use of local well-posedness.} It is well-known that the $b-$equation is locally well-posed, e.g, see \cite{himjmp2014,zhou}. Since $u_0(x):=u(t_0,x)=0$, for all $x\in X$, we conclude that if $u$ is a solution of \eqref{3.1.1} subject to $u(0,x)=u_0$, then $u\equiv0$.
\end{itemize}

Note that if $X=\R$ we can extend the result above for $b=0$. For this choice of $b$, the same procedure above leads us to conclude that $u_x(t_0,x)=0$. Since $u\rightarrow0$ as $|x|\rightarrow\infty$, we get $u(t_0,x)=0$ and the same arguments above prove that $u\equiv0$.

\begin{remark}\label{rem3.2}It is worth mentioning that the $b-$equation has peakon and multipeakon solutions (see \cite{deg,anco}) and they do not belong to the space $ C^0([0,T);H^s(X))\cap C^1((0,T);H^{s-1}(X))$. However, we observe that the peakons all conserve the $L^1(\R)-$norm and, restricting to the Camassa-Holm equation, they also conserve their $H^1(\R)-$norm. Therefore, these solutions can also be analysed in view of the tools introduced in this work. 

In regard to the multipeakon case, not all of them are conservative, but for those having a conserved quantity, the ideas can also be applied to them.
\end{remark}

\begin{remark}\label{rem3.3}
    In \cite{him-cmp} is proved a unique continuation property for the solutions of the CH equation by assuming that the initial datum has a certain exponential decay for larger values of $x$, and that the same happens with the corresponding solution at a latter time. These ideas can also be applied to the $b-$equation (under restrictions to the parameter $b$), as pointed out in \cite{him-cmp} and shown in \cite{henri-na}. Note, however, that these ideas cannot be applied to periodic problems.
\end{remark}

\begin{remark}\label{rem3.4}
    The proof of unique continuation properties of the $b-$equation using local well-posedness of the solutions is nothing but the one proved in \cite{linares} (for both periodic and non-periodic cases), whereas the proof using conserved quantities can be found in \cite{pri-jde}.
\end{remark}

\subsection{The Fornberg-Whitham equation}\label{sec3.2}

Let us now consider the Fornberg-Whitham (FW) equation
\bb\label{3.2.1}
u_t+\f{3}{2}uu_x=\p_x\Lambda^{-2}u,
\ee
which can be obtained from \eqref{3.0.1} by choosing $g(u)=3uu_x/2$ and $G(u)=u$, and use theorem \ref{teo3.1} to prove a unique continuation result for some of its solutions. Again, by example \ref{ex2.1} we see that conditions {\bf G0}--{\bf G3} are satisfied provided that $u\in C^0([0,T);H^s(X))\cap C^1((0,T);H^{s-1}(X))$, with $s>3/2$. Henceforth we assume that $u$ is a non-negative or non-positive solution of the FW equation. 

Let us suppose that we could find $t_0$, $a$ and $b$ such that $\{t_0\}\times[a,b]\subseteq(0,T)\times X$ and $u(t_0,x)=0$, for all $x\in[a,b]$, as well as $u_t(t_0,a)=u_{t}(t_0,b)$. Defining 
$
F_{t}(x)=\p_x\Lambda^{-2}u,
$
we see that as long as $x\in[a,b]$, we have
$
F_t'(x)=\Lambda^{-2}u,
$
and then $F'_{t_0}(x)$ is either non-negative or non-positive whenever $u$ is respectively non-negative or non-positive. 

Under this condition, similarly to the $b-$equation, we have
$$
0=F_{t_0}(b)-F_{t_0}(a)=\int_a^b F'_{t_0}(s)ds,
$$
that is,
\bb\label{3.2.2}
u(t_0,x)=0,\,\,x\in(a,b).
\ee

\begin{itemize}
    \item {\bf Use of conserved quantities.} Equation \eqref{3.2.1} has the conserved quantity \eqref{3.1.4}. Likewise the $b-$equation, it implies the conservation of the $L^1(X)-$norm of a solution $u$ of \eqref{3.2.1} as long as $u$ is either non-negative or non-positive. Proceeding similarly as for the $b-$equation, we conclude that $u\equiv0$.

    \item {\bf Use of local well-posedness.} The FW equation is locally well-posed, see \cite{holmes1,holmes2}. The result follows the same steps as for the $b-$equation and for this reason is omitted.
\end{itemize}

\begin{remark}\label{rem3.5}
The study of unique continuation properties of the FW equation above assumes that the solution does not change its sign. Although the FW has this sort of solutions (such as its peakon solution), it is unclear whether such a information could be inferred for solutions emanating from a suitable initial datum. This is a quite deep question and a positive answer to it would lead to the proof of existence of global solutions. In fact, \cite[Proposition 1]{haz} says that solutions of the FW equation blows-up in finite time if and only if the slope of the solution cannot be bounded from below when time approaches to a certain value. On the other hand, assuming that a certain initial datum satisfies the condition $u_0-u_0''\geq0$ (that in particular, would imply $u_0\geq0$), following the same steps of \cite[Theorem 3.5]{const1998-1} we can demonstrate that the corresponding solution is bounded from below at any time it exists.  Despite the relevance of the question, an investigation on it is out of the scope of the present work.
\end{remark}

\subsection{The potential Camassa-Holm equation}\label{sec3.3}

Here we apply theorems \ref{teo2.1} and \ref{teo2.3} to establish unique continuation results for non-periodic solutions of the equation
\bb\label{3.3.1}
u_{t}-u_{txx}=\f{1}{2}(3u_x^2-2u_xu_{xxx}-u_{xx}^2),
\ee
that was discovered by Novikov \cite{nov} and we shall refer to it as potential CH equation. 

Equation \eqref{3.3.1} can be rewritten as
\bb\label{3.3.2}
u_{t}-\f{1}{2}u_x^2=\Lambda^{-2}\Big(u_x^2+\f{1}{2}u_{xx}^2\Big),
\ee
so that it can be obtained from \eqref{3.0.1} by taking $A(u)=u_x\p_x$ and $g(u)=\Lambda^{-2}(u_x^2+u_{xx}^2/2)$. Moreover, taking  $u\in C^0([0,T);H^s(\R))\cap C^1((0,T);H^{s-1}(\R))$, with $s>5/2$, we see that {\bf G0}--{\bf G3} are satisfied with ${\cal B}=H^s(X)$, $s>5/2$.

Suppose that we could find $t_0$, $a$ and $b$ such that $u_t(t_0,x)=u_{t}(t_0,x)=0$, $x\in[a,b]$. By \eqref{3.3.2} we would then have
$$
g(u(t_0,x))=\Lambda^{-2}\Big(u_x^2+\f{1}{2}u_{xx}^2\Big)(t_0,x)=(u_{t}-\f{1}{2}u_x^2)(t_0,x)=0.
$$

Since $\Lambda^{-2}(u_x^2+u_{xx}^2/2)(t_0,x)=0$, we conclude that $u_x(t_0,\cdot)=0$, and then $u(t_0,x)=const$, $x\in \R$. Since $u$ vanishes at infinity, we conclude that
\bb\label{3.3.3}
u(t_0,x)=0,\quad x\in\R.
\ee

\begin{itemize}
    \item {\bf Use of conserved quantities.} It is known, see \cite[Lemma 4.1]{tu}, that 
$$
{\cal H}(t)=\int_\R (u_x^2+u_{xx}^2)(t,x)dx
$$
is a conserved quantity for the equation, from which we conclude that $u_x(t,\cdot)=0$, for all $t\in[0,T)$, because ${\cal H}(t_0)=0$. Since $u\rightarrow0$ as $|x|\rightarrow\infty$, we conclude that $u\equiv0$.
    
\item {\bf Use of local well-posedness.} For $u\in H^s(\R)$, $s>5/2$ it is known that \eqref{3.3.1} is locally well posed, see \cite[Section 3]{tu}. Therefore, like in the previous cases, \eqref{3.3.3} implies that $u$ vanishes everywhere.
\end{itemize}

\subsection{The $\pi-$Camassa-Holm system}\label{sec3.4}

We now use theorems \ref{teo2.2} and \ref{teo2.4} to establish unique continuation results for solutions of the $\pi-$Camassa-Holm system \cite{koh}
\bb\label{3.4.1}
\left\{\ba{l}
m_t+um_x+2u_xm=-\pi(\rho)\rho_x,\\
\\
\pi(\rho)_t=-(\pi(\rho)u)_x,
\ea
\right.
\ee
where $m=u-u_{xx}$ and 
$$
\pi(\rho)=\rho-\int_\s \rho dx.
$$

We can rewrite \eqref{3.4.1} as the following system of evolution equations
\bb\label{3.4.2}
\left\{\ba{l}
\ds{u_t+uu_x=-\p_x\Lambda^{-2}\Big(u^2+\f{u_x^2}{2}+\f{\pi(\rho)^2}{2}\Big)},\\
\\
\ds{\pi(\rho)_t+u\rho_x=-\pi(\rho)u_x}.
\ea
\right.
\ee

Noticing that $\p_x\pi(\rho)=\rho_x$, comparing \eqref{3.4.2} with \eqref{1.0.2}, we conclude that $\uu=(u,\pi(\rho))$, and
$$
A(\uu)\uu=\begin{pmatrix}
u\p_x & 0\\
\\
0 & u\p_x
\end{pmatrix}
\begin{pmatrix}
u\\
\\
\pi(\rho)
\end{pmatrix},\quad
g=\begin{pmatrix}
\ds{-\p_x\Lambda^{-2}\Big(u^2+\f{u_x^2}{2}+\f{\pi(\rho)^2}{2}\Big)}\\
\\
-\pi(\rho)u_x
\end{pmatrix}=:\begin{pmatrix}
g^1(\uu)\\
\\
g^2(\uu)
\end{pmatrix}.
$$

In addition, we see that {\bf G0}--{\bf G3} are satisfied whenever $(u,\rho)\in C^0([0,T);H^s(\s)\times H^{s-1}(\s))\cap C^1((0,T);H^{s-1}(\s)\times H^{s-2}(\s))$, $s>5/2$, and ${\cal B}=H^s(\s)\times H^{s-1}(\s))$, $s>5/2$.

Suppose that conditions {\bf C3} and {\bf C4} of theorem \ref{teo2.2} hold and consider $F^1_{t_0}(x)=g^1(\uu(t_0,x))$, that is,
\bb\label{3.4.3}
F^1_{t_0}(x)=-\p_x\Lambda^{-2}\Big(u^2+\f{u_x^2}{2}+\f{\pi(\rho)^2}{2}\Big)(t_0,x).
\ee

From \eqref{3.4.3}, \eqref{3.4.2} and using {\bf C4}, we conclude that $F_{t_0}^1(a)=F_{t_0}^1(b)=0$. Moreover, the fundamental theorem of calculus reads
$$
0=F_{t_0}^1(b)-F_{t_0}^1(a)=\int_{a}^b\Lambda^{-2}\Big(u^2+\f{u_x^2}{2}+\f{\pi(\rho)^2}{2}\Big)(t_0,s)ds,
$$
forcing us to conclude that $u(t_0,\cdot)=0$, as well as $\pi(\rho)(t_0,\cdot)=0$. 

\begin{itemize}
    \item {\bf Use of conserved quantities.} Since
$
{\cal H}(t)=\|u(t,\cdot)\|^2_{H^1(\s)}+\|\pi(\rho)(t,\cdot)\|^2_{L^2(\s)}
$
is a conserved quantity for \eqref{3.4.1}, see \cite[Lemma 3.1]{ma}, and ${\cal H}(t_0)=0$, we conclude that $u\equiv0$, as well as $\pi(\rho)\equiv0$.
    
    \item {\bf Use of local well-posedness.} In \cite[Corollary 3]{koh}, \cite[Theorem 2.1]{ma} it was shown that \eqref{3.4.1} is locally well-posed provided that the initial data $(u_0,\rho_0)\in H^{s}\times H^{s-1}$, $s>5/2$. Since $u(t_0,\cdot)=0$ and $\pi(\rho)(t_0,\cdot)=0$, theorem \ref{teo2.4} implies that $u\equiv0$, as well as $\pi(\rho)\equiv0$.
\end{itemize}

\section{Application to the generalised Boussinesq equation}\label{sec4}

We now apply our machinery to equation \eqref{1.0.4}, that is quite illustrative inasmuch as it shows how our ideas can be employed to a class of equations larger than that one might think at first glance. Our goal in this section is to prove the following result.

\begin{theorem}\label{teo4.1}
Assume that $f$ is a non-negative and smooth function, vanishing only at $0$, and $s\geq0$. If a solution $u\in C^2([0,T);H^s(\R)\cap L^\infty(\R))$ of \eqref{1.0.4} vanishes on a non-empty and open set $\Omega\subseteq [0,T)\times\R$, then $u\equiv0$.
\end{theorem}

Before proving the theorem above, let us first explore a bit more \eqref{1.0.4} and make some observations that will be useful for our purposes. 

Let us assume that $u\in C^2([0,T);H^s(\R)\cap L^\infty(\R))$, for some $T>0$, and $s\geq0$. Apparently \eqref{1.0.4} does not seem to belong to the class \eqref{1.0.2}. Let us define $v:=u_t$. From \eqref{1.0.4} we obtain $v_{t}=\p_x^2 f(u)+v_{txx}$, that is $\Lambda^{2}v_t=f(u)$. Inverting the operator $\Lambda^{-2}$, using the identity $\p_x^{2}\Lambda^{-2}=\Lambda^{-2}-1$, and the constraint relating $u$ and $v$, we can rewrite \eqref{1.0.4} as the system
\bb\label{4.0.1}
\left\{\ba{lcl}
v_t&=&\Lambda^{-2}f(u)+f(u),\\
\\
u_t&=&v.
\ea\right.
\ee

Let us assume that $\Omega$ is a non-empty open set in which $u$ vanishes. Then we can find numbers $t^\ast$, $a$, $b$, and $\eps>0$ such that ${\cal R}=I_\eps\times[a,b]\subseteq\Omega$, where $I_\eps=[t^\ast-\eps,t^\ast+\eps]$. We note that for each $t\in I_\eps$ fixed, conditions ${\bf C3}$ and ${\bf C4}$ in theorem \ref{teo2.2} are satisfied for the pair $(u,v)$.

We recall that $f$ is a smooth, non-negative function, vanishing only at $0$. Defining $F_t(x)=\Lambda^{-2}f(u(t,x))$, we straightforwardly conclude that $F_t(x)\geq0$ for all $x$ and it vanishes if and only if $u(t,\cdot)\equiv0$. 

As long as $(t,x)\in{\cal R}$, the first equation in \eqref{4.0.1} yields
$$
0=v_t-f(u(t,x))=\Lambda^{-2}f(u(t,x))=\f{1}{2}\int_\R e^{-|x-y|}f(u(t,y))dy,
$$
which implies that $u(t,x)=0$, for all $t\in I_\eps$ and $x\in\R$.

\begin{proof}Let us fix $t_0\in (t^\ast-\eps,t^\ast+\eps)$ and define $u_0(\cdot)=u(t_0,\cdot)$ and $u_1(\cdot)=u_t(t_0,\cdot)$. Since  $u\in C^2([0,T);H^s(\R)\cap L^\infty(\R))$, $s\geq0$, we then conclude that $u_0,u_1\in H^s(\R)\cap L^\infty(\R)$. Moreover,  $(u_0,u_1)(x)=(0,0)$ and then the solution of the problem
\bb\label{4.0.2}
\left\{\ba{lcl}
u_{tt}=\p_x^2f(u)+u_{ttxx},\\
\\
u(t_0,x)=u_0(x),\\
\\
u_t(t_0,x)=v_0(x),
\ea\right.
\ee
vanishes identically, since \cite[Theorem 1]{moli} assures that \eqref{4.0.2} is well-posed and its solution depends continuously on the initial datum.

Let us now give a more elaborate proof of the same fact, but using conserved quantities. To this end, we introduce new functions and use a conserved quantity found in \cite{moli} to re-obtain  our conclusion.

Substituting $z_x:=u_t$ and $w_x:=u$, see \cite[page 1070]{moli}, into \eqref{1.0.4} we obtain $z_{tx}=\p_x^2 F(w_x)+z_{txxx}$ and $w_{tx}=z_x$. Integrating these two equations and noticing that $(z,w)\rightarrow0$ at infinity (note that both belong to $H^{s+1}(\R)$ and $s\geq0$), we conclude that $(z,w)$ satisfies 
\bb\label{4.0.3}
\left\{\ba{lcl}
z_t&=&\p_x(1-\p_x^2)^{-1}F(w_x),\\
\\
w_t&=&z,
\ea\right.
\ee
which is clearly of the type \eqref{1.0.2}.

From \cite[Equation 4.7]{moli} (see also \cite[Theorem 4]{moli}), \eqref{4.0.3} has the conserved quantity
\bb\label{4.0.4}
{\cal H}(t)=\f{1}{2}\|z(t,\cdot)\|^2_{H^1(\R)}+\int_\R P(u(t,x))dx,
\ee
where 
\bb\label{4.0.5}
P(x)=\int_0^xf(s)ds.
\ee

Since $f$ is non-negative and $f(x)=0$ if and only if $x=0$, we then conclude that 
$$
\int_\R P(u(t,x))dx\geq0
$$
and it vanishes if and only if $u(t,x)=0$, $x\in\R$.
 
For $t_0\in I_\eps$ and the conditions on $z$, we conclude that $z(t_0,x)=0$. On the other hand, we have already concluded that $u(t_0,\cdot)=0$, which means the quantity ${\cal H}(t)$ vanishes at $t=t_0$. This implies that for any other $t$, we have
$$0={\cal H}(t_0)={\cal H}(t)\geq\int_\R P(u(t,x))dx\geq0,$$
and consequently, $f(u(t,x))=0$ and $u(t,x)=0$ as well, which implies that $u\equiv0.$
\end{proof}

\section{The modified Euler-Poisson system}\label{sec5}

Noting the presence of the Helmholtz operator $\Lambda^2=1-\p_x^2$ in the second equation of the system \eqref{1.0.5}, we can rewrite it in the formal and more convenient form
\bb\label{5.0.1}
\left\{
\ba{l}
\rho_t+(u\rho)_x=0,\\
\\
u_t+uu_x+\p_x\Lambda^{-2}\rho=0.
\ea
\right.
\ee

The set of equations \eqref{5.0.1} can be obtained from the Euler-Poisson system of equations after linearisation of one of its equations, see \cite[Introduction]{tiglay} for further details. Moreover, by taking $\uu=(u,\rho)$, it can put into the form \eqref{1.0.2}, with
$$
A(\uu)\uu=\begin{pmatrix}
u\p_x & 0\\
\\
0 & u\p_x
\end{pmatrix}
\begin{pmatrix}
u\\
\\
\rho
\end{pmatrix},\quad
g=\begin{pmatrix}
\ds{-\p_x\Lambda^{-2}(\rho)}\\
\\
-\rho u_x
\end{pmatrix}=:\begin{pmatrix}
g^1(\uu)\\
\\
g^2(\uu)
\end{pmatrix}.
$$

Furthermore, it is integrable in view of its bi-Hamiltonian structure \cite[Theorem 3]{tiglay}, and its Hamiltonians are the functionals
\bb\label{5.0.2}
\ba{lcl}
{\cal H}_1(t)&=&\ds{\int_\R(\rho u^2+(\Lambda^{-2}\rho)^2+(\Lambda^{-2}\rho)^2)dx},\\
\\
{\cal H}_2(t)&=&\ds{\int_\R u\rho dx}.
\ea
\ee

It is also worth mentioning a third conserved quantity for the modified Euler-Poisson system, given by
\bb\label{5.0.3}
{\cal H}(t)=\int_\R \rho dx.
\ee

The quantities in \eqref{5.0.2}--\eqref{5.0.3} are all invariants along time for solutions with enough decaying as $|x|\rightarrow\infty$. In particular, they are conserved for the sort of solutions considered in \cite{him-dcds}.

\begin{definition}\label{def5.1}
    A pair $(u,\rho)$ is said to be a conservative solution for the system \eqref{1.0.5} if it is a solution of the system \eqref{5.0.1} and has the functional \eqref{5.0.3} as a conserved quantity.
\end{definition}

We observe that as long as $\rho(t,\cdot)$ does not change it sign (that is, $\rho(t,x)\geq0$ or $\rho(t,x)\leq0$, for all $(t,x)$ such that the solution exists), then the functional in \eqref{5.0.3} is nothing but $\|\rho(t,\cdot)\|_{L^1(\R)}$.

Our main result in this section is the following.

\begin{theorem}\label{teo5.1}
    Assume that $(u,\rho)$ is a conservative, strong solution of the system \eqref{1.0.5}  defined on $[0,T)\times\R$, such that $\rho$ is non-negative. If there exists an non-empty, open set $\Omega\subseteq[0,T)\times\R$, such that $(u,\rho)\big|_\Omega\equiv0$, then $\rho\equiv0$ and $u$ is a solution of the equation $u_t+uu_x=0$.
\end{theorem}

Let us show that system \eqref{5.0.1} admits a solution $(u,\rho)$, with $\rho(t,x)\geq0$ for all $[0,T)\times\R$. To this end, let us define $q(t,x)$ through the ODE
\bb\label{5.0.4}
\left\{
\ba{lcl}
q_t(t,x)&=&u(t,q(t,x)),\\
\\
q(0,x)&=&x.
\ea
\right.
\ee

Under mild conditions on $u$ we can guarantee the existence and uniqueness of a unique solution for the system above, for each fixed $x$. Therefore, varying $x$ we then have a map $(t,x)\mapsto q(t,x)$, satisfying $q(0,x)=x$. We will show that depending on the regularity of the solutions of the system \eqref{5.0.1}, for each fixed $t$, the function $q(t,\cdot)$ is an increasing diffeomorphism from $\R$ into itself.

\begin{theorem}\label{teo5.2}
    Assume that $(u_0,\rho_0)\in H^3(\R)\times H^{2}(\R)$ and $\rho_0$ is a non-negative function. Then the corresponding solution of \eqref{5.0.1} subject to $(u(0,x),\rho(0,x))=(u_0(x),\rho_0(x))$ is a conservative solution for the system \eqref{1.0.5} and the property of being non-negative persists for $\rho(t,x)$ as long as it exists.
\end{theorem}

\begin{proof}
Since $(u_0,\rho_0)\in H^3(\R)\times H^{2}(\R)$, \cite[Theorem 1.3]{him-dcds} guarantees that $u\in C^0([0,T),H^3(\R))$, for some $T>0$. Proceeding as in \cite[Theorem 1]{tiglay}, we can also conclude that $u\in C^1([0,T),H^2(\R))$, wherefrom we conclude that $u$ and $u_x$ are both Lipschitz, bounded and Lipschitz with respect to $x$ and continuous in $t$.

    By letting $x$ varies in \eqref{5.0.4}, we can differentiate it with respect to $x$, to obtain
    $$
    \left\{
\ba{lcl}
\p_tq_x(t,x)&=&u_x(t,q(t,x))q_x,\\
\\
q_x(0,x)&=&1,
\ea
\right.
    $$
whose solution is 
$$\ds{q_x(t,x)=e^{\int_0^t u_x(\tau,x)d\tau}>0.}$$
This shows that $q(t,\cdot)$ is an increasing diffeomorphism from $\R$ into its image, $t\in[0,T_0)$, for some $T_0\in(0,T)$. The proof that $T_0=T$ and $q(t,\R):=\{q(t,x),\,x\in\R\}=\R$, $t\in[0,T)$, follows the same steps as \cite[Theorem 3.1]{const-f}, and for this reason is omitted.

Now we evaluate $\rho$ at $(t,q(t,x))$ and use the first equation in \eqref{5.0.1} and \eqref{5.0.4} to arrive at
$$
\f{d}{dt}\rho(t,q)=\rho_t(t,q)+\rho_x(t,q)q_t=-\rho(t,q)u_x(t,q),
$$
whose integration gives
\bb\label{5.0.5}
\rho(t,q(t,x))=\rho_0(x)e^{-\int_0^tu_x(\tau,q(\tau,x))d\tau}.
\ee

Since $q(t,x)$ is a diffeomorphism from $\R$ into itself and $\rho_0$ is non-negative, so do is $\rho(t,x)$.
\end{proof}

{\bf Proof of theorem \ref{teo5.1}.} For each $t\in[0,T)$, let us define
\bb\label{5.0.6}
F_t(x)=(\p_x\Lambda^{-2}\rho)(t,x),
\ee
and thus, $F_t(\cdot)\in C^1(\R)$. 

From the second equation in \eqref{5.0.1} and any $t\in(0,T)$, we can also express the function in \eqref{5.0.6} by
\bb\label{5.0.7}
F_t(x)=-(u_t+uu_x)(t,x),
\ee
hence $F_t(x)=0$, provided that $(t,x)\in\Omega$. Therefore, we can find $\{t_0\}\times[a,b]\subseteq \Omega$, with $a<b$, such that $F_{t_0}(x)=0$, $a\leq x\leq b$. In particular, $F_{t_0}(a)=F_{t_0}(b)=0$. Theorem \ref{teo5.1} is a straightforward corollary of the following result.

\begin{theorem}\label{teo5.3}
    Under the conditions in theorem \ref{teo5.1}, if there exists $a<b$ and $t_0\in\R$ such that $\{t_0\}\times[a,b]\subseteq [0,T)\times\R$ and $F_{t_0}(x)=0$, $x\in[a,b]$, then $\rho\equiv0$.
\end{theorem}

For $x\in[a,b]$, we have $F_{t_0}'(x)=(\Lambda^{-2}\rho)(t_0,x)$, and since it is non-negative, the fact that $F_{t_0}(a)=F_{t_0}(b)$ forces $F'_{t_0}(x)=0$ within $(a,b)$. As a result, fixed $x_0\in(a,b)$, we have
$$
\int_\R e^{-|y-x_0|}\rho(t_0,y)dy=0,
$$
that is, $\rho(t_0,x)=0$, $x\in\R$. As a consequence, the functional \eqref{5.0.3} vanishes at $t=t_0$. Since it is constant and $\rho$ is non-negative, then $\rho$ must be $0$. \hfill$\square$

\begin{remark}\label{rem5.1}
    Formula \eqref{5.0.5} is equivalent to \cite[Equation (70)]{him-dcds}, and jointly with \cite[Equation (3)]{him-dcds} they imply that the solution $n$ (which is nothing but our $\rho$ in \eqref{5.0.1}) considered in \cite{him-dcds} is non-negative. 
\end{remark}

\begin{remark}\label{rem5.2}
    Replacing \eqref{5.0.3} by the functional ${\cal H}_1(t)$ (see \eqref{5.0.2}) in definition \ref{def5.1} we would still have theorem \ref{teo5.1} valid. 
\end{remark}

\begin{remark}\label{rem5.3}
    Theorem \ref{teo5.1} would still hold if we assume that $\rho$ is non-positive, but differently from Remark \ref{rem5.2}, it would not be possible to prove it replacing \eqref{5.0.3} by ${\cal H}_1(t)$ in definition \ref{def5.1}. While \eqref{5.0.3} would still be, up to a sign, equivalent to the $L^1(\R)-$norm of $\rho$, the function ${\cal H}_1$ would no longer be non-negative. 
\end{remark}

\begin{remark}\label{rem5.4}
  \cite[Theorem 1.1]{him-dcds} tells us that if the initial datum $(u_0,\rho_0)$ satisfies a certain exponential decay, $\rho_0$ is non-negative, and if the corresponding solution emanating from \eqref{5.0.1}, subject to such an initial condition, has the same sort of decay at a latter time, then $\rho\equiv0$ and $u$ is a solution of the equation $u_t+uu_x=0$.  Our theorem \ref{teo5.1} gives a different proof for the same fact, but without using the exponential decay condition: we use the fact that the solution vanishes on an open set instead. Moreover, the same fact can also be concluded as long as we assume $\rho$ is non-positive.
\end{remark}

\begin{remark}\label{rem5.5}
 With little effort we can proof a periodic version of theorems \ref{teo5.1} -- \ref{teo5.3}. In fact, their proofs are essentially the same replacing $\R$ by $\s\approx[0,1)$. The only technical detail to be taken is choosing $[a,b]\subseteq[0,1)$ in  theorem \ref{teo5.3}.
\end{remark}

\begin{remark}\label{rem5.6}
Unlike \cite{him-dcds}, we can give a unique continuation result for periodic solutions of \eqref{5.0.1} in view of Remark \ref{rem5.5}. In fact, the results proved in \cite{him-dcds} cannot be applied for periodic solutions since a sine qua non condition in \cite[Theorem 1.1]{him-dcds} is an exponential decay behavior of the solutions as $x\rightarrow\infty$, which shows a significant difference with our approach.
\end{remark}

\section{Discussion}\label{sec6}

The results reported in this paper have two different motivations: the one related to conserved quantities is based on the concepts of conservation of physical quantities, such as mass or energy, whereas the other, related to existence of solutions, is concerned with qualitative aspects of solutions for the dynamical system \eqref{1.0.2}. In the first case the independent variables $t$ and $x$ have the meaning of time and space, respectively.

Equations of the type \eqref{1.0.2} coming from Physics usually {\it do have} conserved densities, say $h([\uu])$, which gives rise to the conserved quantity
$$
{\cal H}(t)=\int_\Omega h([\uu])dx,
$$
associated to a solution $\uu$. These equations usually describe wave propagating in a media, such as water, and very often, a conserved density refers to mass or kinect energy densities, that are described by non-negative functions $h([\uu])$. If a mass density vanishes, then such fact implies on the non-existence of a media in which a wave could propagate. On the other hand, if the density refers to kinect energy, its vanishing means that the there is no perturbation in a medium, that is, we do not have any event causing a phenomenon. These observations are the main ideas behind the results developed in the subsection \ref{sec2.3}.

We note that while {\bf C0}, {\bf C1}, {\bf C3}, and {\bf C4} are properties of the {\it solution of the system}, {\bf C2}, {\bf C5}, {\bf C6}, and {\bf C7} are conditions arising from the {\it system/Cauchy problem}. Moreover, these attributes are used to build our results from the two perspectives aforementioned, that despite being conceptually rather different, both have a common crucial ingredient: the existence of a value $t_0$ for which a solution $u$ vanishes. If one is successful in finding such a value, one can then prove that the properties stated in conditions {\bf C1}; or {\bf C3} and {\bf C4}, persist, since the solution necessarily vanishes and these conditions are satisfied for the identically zero solution. The main difference between the two approaches is the way we extend to all values of $t$ the information we got to $t_0$. This can be done in the following forms:
\begin{itemize}

\item From a physical point of view, in case we are able to find a time $t_0$ for which the conserved quantity vanishes, we then conclude that it vanishes for each $t$ such that the solution exists, in view of its time invariance. Physically speaking, usually the conserved density vanishes if and only if the field variable $\uu$ vanishes.

\item From a mathematical perspective the results developed in the subsection \ref{sec2.4} are concerned with well-posed Cauchy problems. Remark \ref{rem2.5} shows that if we can find a value of $t$ for which the solution vanishes, it then vanishes everywhere in view of uniqueness. 

\end{itemize}

The two different views above are connected in the following way: if the mass or kinect energy of a physical system vanishes, then we do not have any non-zero initial data provoking a modification in the system. Mathematically speaking, we have a zero initial data to the Cauchy problem \eqref{2.2.1}. Still from mathematical eyes, the mass or kinect energy associated to conserved quantities very often corresponds to some norms in suitable Banach spaces.

Note that while the existence of a conserved quantity or the local well-posedness is a crucial ingredient to extend to any value of $t$ the same property known to $t_0$, it is the map $g$ in \eqref{1.0.2} that gives the information that the solution vanishes at some specific time, see conditions ${\bf C2}$ or ${\bf C6}$ in theorems \ref{teo2.1}--\ref{teo2.4}.

The aforementioned results have as an immediate corollary the following: given a solution $\uu$ of the system \eqref{1.0.2} satisfying certain conditions, if it vanishes on an open set $\Omega\subseteq[0,T)\times X$, then it necessarily vanishes identically. Hence, such a solution is global. This consequence is explicitly stated in the corollary of theorem \ref{teo2.1}, but an analogous result could be proved as a corollary for theorem \ref{teo2.3}. Moreover, theorem \ref{teo3.1} is a manifestation of this fact for scalar equations.

In \cite{linares} was proved a unique continuations result for the $b-$equation ($b\in[0,3]$) and, more generally, for equations of the type \eqref{3.0.1}, with functions $g$ and $G$ depending only on $u$ or $u_x$, or both, and $G$ is non-negative and vanishes only when $u=0$, see \cite[Theorem 1.6]{linares}. Following our language, the results in \cite{linares} were established using local well-posedness. Our theorem \ref{teo3.1} in particular, and the results presented in section \ref{sec2} in general, generalise these results not only for equations depending on higher order derivatives, but also for systems of equations.

Recently, conserved quantities have been used for establishing unique continuation properties for solutions of the Dullin-Gottwald-Holm equation \cite{freire-jpa} (see also its correction \cite{freire-cor}); the BBM equation \cite{raspa-mo}; 0-Holm-Staley equation \cite{pri-jde,nilay}; a generalisation of the CH equation \cite{igor-jmp}; and a geometrically integrable equation \cite{nazime}. All of these references are concerned to equations, while in \cite{freire-bi} the conservation of norms of solutions of a bi-Hamiltonian system (that has the CH as a particular case) was used to establish unique continuation of its solutions, giving a strong indication that the results (and approaches) considered in \cite{linares,freire-jpa,freire-cor} could be generalised and extended to systems of differential equations or for equations like \eqref{3.0.1} depending on a finite jet of the solution $\uu$.

Most of the works mentioned above are concerned with, or motivated by, hydrodynamic equations, and some of our examples as well. Therefore, it is natural that one wonders if our results could be applied to the KdV equation, that is a quite famous hydrodynamic model. It can be written as
$$
u_t=\al uu_x+u_{xxx}=:A_1 u,
$$
where $A_1=\al u\p_x+\p_x^3$, and $\al\neq0$ is constant that we do not impose any restriction. Comparing the KdV equation with \eqref{1.0.2} we see that $A=A_1$ and $g(u)\equiv0$. Therefore, none of the conditions {\bf C2} or {\bf C6} are satisfied and our results cannot be applied to it. 

Last but not least, it is worth mentioning that our machinery enable us to establish unique continuation properties for the solutions of \eqref{1.0.4}, which as far as the author knows, has not been proved yet.

\section{Conclusion}\label{sec7}

In this paper we establish conditions for proving unique continuation results for the arbitrary system \eqref{1.0.1} using two different points of view. Our results are then applied to some important equations in hydrodynamics, wave propagation and integrable systems. Also, we proved a unique continuation result for the generalised Boussinesq equation \eqref{1.0.4} using the tools developed for tackling the class \eqref{1.0.1}.

\section*{Acknowledgements}

I am grateful to CNPq (grant nº 310074/2021-5) and FAPESP (grant nº 2020/02055-0) for financial support. I must thank the anonymous comments I received, that led me to the references \cite{him-dcds,tiglay}. I am deeply indebted to Professor G. Misiolek for sharing with me the paper \cite{him-dcds}. Most of this work was developed and written while I was a staff member of Universidade Federal do ABC (UFABC), which I gladly thank for all the support I had. Finally, I want to express my gratitude to the Department of Mathematical Sciences of Loughborough University (UK), where the results of section 5 and a moderate revision in the text was carried out.

\end{document}